\documentclass[a4paper,11pt,3p]{article}
\usepackage[margin=1in]{geometry}
\usepackage{amsfonts,amsmath,amssymb,mathrsfs,amsthm,amscd,enumitem,multicol,multienum}
\usepackage{graphics,multirow}
\usepackage{algorithm}
\usepackage{tabularx}
\usepackage{chemfig}
\usepackage[version=4]{mhchem}
\usepackage{bm}
\usepackage{latexsym}
\usepackage[T1]{fontenc}
\usepackage[utf8]{inputenc}
\usepackage[colorlinks=true, linkcolor=black, anchorcolor=black, citecolor=black, filecolor=black, urlcolor=black]{hyperref}

\newtheorem{thm}{Theorem}[section]
\newtheorem{lemma}[thm]{Lemma}
\newtheorem{cor}[thm]{Corollary}
\newtheorem{prop}[thm]{Proposition}
\newtheorem{defin}[thm]{Definition}

\numberwithin{equation}{section}

\renewcommand{\epsilon}{\varepsilon}

\begin{document}
	\begin{center}
		{\large Global existence and uniform boundedness of the classical solutions for the system of multi-species transport with mass control }

	\bigskip
	
	\bigskip
\end{center}
\begin{minipage}[h]{0.45\columnwidth}
	\begin{center}
		Nibedita Ghosh\\
		{\small Department of Mathematics, \\IIT Kharagpur}\\
		{\small WB 721302, India}\\
		{\small e-mail: nghosh.iitkgp@gmail.com}
	\end{center}
\end{minipage}\hspace{1cm}
\begin{minipage}[h]{0.45\columnwidth}
	\begin{center}
		Hari Shankar Mahato\\
		{\small Department of Mathematics, \\IIT Kharagpur}\\
		{\small WB 721302, India}\\
		{\small e-mail: hsmahato@maths.iitkgp.ac.in}
	\end{center}
\end{minipage}

\bigskip
\hrule

	\bigskip
	
	\textbf{Abstract.}  The goal of this work is to establish the global existence of nonnegative classical solutions in all dimensions for a system of highly nonlinear reaction-diffusion equations. We address the case for different diffusion coefficients and the system of reversible reactions with non-homogeneous Neumann boundary conditions. The systems are assumed to satisfy only the mass control condition and to have locally Lipschitz nonlinearities with arbitrary growth. The key aspect of this work is that we didn't assume that the diffusion coefficients are close to each other. We utilize the duality method and the regularization of the heat operator to derive the result. We also illustrate the global in time bounds for the solutions. The application includes concrete corrosion in sewer pipes or sulfate corrosion in sewer pipes.
	
	\bigskip
	{\textbf{Keywords:} } reaction-diffusion systems, classical solutions, global existence, mass control, uniform boundedness. \\

	{\textbf{AMS subject classifications: }} 35A01, 35K57, 35K58, 47J35, 80A32, 82C70
	\bigskip
	\hrule
	
	\section{Introduction}
	Reaction-diffusion systems appears frequently in physical, chemical or biological models and the study of the global well-posedness for such systems is a subject of interest since last four decades, see, e.g., \cite{rothe1984global, hollis1987global, amann1985global, conway1978large, fitzgibbon1997stability, morgan1989global} and many references therein. In the present paper, we study the reaction-diffusion systems in the context of reversible chemical reactions. Our work can be motivated in the following way: We consider $\Omega\subset\mathbb{R}^n(n\ge2)$ be a bounded domain with sufficiently smooth boundary $\partial\Omega$. We study a multi-species multi-reaction model with $I\in\mathbb{N}$ number of mobile species(dissolved) $M_i$ present in the carrier substance in the domain $\Omega$ and $J\in\mathbb{N}$ number of reactions among the mobile species. We take into account that these mobile species are diffuse and react via the reversible reactions given by 
	\begin{equation}\label{eqn:re}
		\mu_{1j}M_1+\mu_{2j}M_2+\cdots\mu_{Ij}M_I\rightleftharpoons \nu_{1j}M_1+\nu_{2j}M_2+\cdots\nu_{Ij}M_I \qquad \text{ in }\quad \Omega,
	\end{equation}
	where $j=1,2,\cdots, J$. Here we assume that all the stoichiometric coefficients $\mu_{ij}, \nu_{ij}$ are non-negative integers. The vector of the mobile species concentrations is denoted by $u=(u_1, u_2,. . .,u_I)$. The mobile species are transported by diffusion only which is modeled by the Fick's law. We didn't consider any fluid velocity that is there is no advection. Therefore the flux vector is given by $\vec{j}_i=-\bar{D}_i\nabla u_i$, where $\bar{D}_i$ is the diffusive matrix. The diffusive matrix is defined by $\bar{D}_i=diag(D_i, D_i, . . ., D_i)$. Therefore the diffusive flux is $\vec{j}_i=-diag(D_i, D_i, . . ., D_i)\nabla u_i=-D_i\nabla u_i$, where $D_i$ is a positive scalar and different for $I$ number of mobile species. The reaction amongst the $I$ number of mobile species is modeled by mass action kinetics. Hence the $j-$th reaction rate term is given by 
	\begin{align*}
		R_j(u)=R_j^f(u)-R_j^b(u),
	\end{align*}
	where
	\begin{align*}
		R_j^f(u)=\text{ forward reaction rate term }=k_j^f\prod_{m=1}^{I}(u_m)^{\mu_{mj}}
	\end{align*}
	and 
	\begin{align*}
		R_j^b(u)=\text{ backward reaction rate term }=k_j^b\prod_{m=1}^{I}(u_m)^{\nu_{mj}}
	\end{align*}
	where $k_j^f$ and $k_j^b$ are the forward reaction rate constant and backward reaction rate constant, respectively. Therefore the reaction rate term for the $i-$th mobile species is represented by
	\begin{align*}
		(SR(u))_i=\sum_{j=1}^{J}s_{ij}R_j(u)=\sum_{j=1}^{J}s_{ij}\left(k_j^f\prod_{m=1}^{I}(u_m)^{\mu_{mj}}-k_j^b\prod_{m=1}^{I}(u_m)^{\nu_{mj}}\right),
	\end{align*} 
	where $S=(s_{ij})_{1\le i\le I, 1\le j\le J}$ is the stoichiometric matrix with entries $s_{ij}=\nu_{ij}-\mu_{ij}$. We set $S:=[0,T)$ as the time interval for a $T>0$, then our diffusion-reaction model for $I$ number of mobile species is given by
	\begin{subequations}
		\begin{align}
			\frac{\partial u_i}{\partial {t}} + \nabla.(-{D_i}\nabla u_i)&=f_i(u) \; \text{ in } S\times \Omega,	\label{eqn:M11}\\
			-{D_i}\nabla u_i.\vec{n}&=b_i \; \text{ on } S \times \partial\Omega, \label{eqn:M12}\\
			u_i(0,x)&=u_i^0(x) \; \text{ in } \Omega,\label{eqn:M13} 
		\end{align}	
	\end{subequations}
	where the source term 
	\begin{align}\label{eqn:st}
		f_i(u)=(SR(u))_i=
		\sum_{j=1}^{J}s_{ij}\left(k_j^f\prod_{m=1}^{I}(u_m)^{\mu_{mj}}-k_j^b\prod_{m=1}^{I}(u_m)^{\nu_{mj}}\right).
	\end{align}
	We represent the system of equations $\eqref{eqn:M11}-\eqref{eqn:M13}$ by $(\mathcal{P})$.
	
	If we choose $\lambda=\max\{\sum_{m=1}^{I}\mu_{mj}, \sum_{m=1}^{I}\nu_{mj}\}$, we can control the growth of the source term. We consider $u_1^2, u_2^2, . . ., u_I^2$ together with their respective weights $\mu_{1j}, \mu_{2j}, . . ., \mu_{Ij}$ and $\nu_{1j}, \nu_{2j}, . . ., \nu_{Ij}$, respectively. Then using weighted G.M. $\le$ weighted A.M.,
	 we arrive at
	\begin{align*}
		\prod_{m=1}^{I}(u_m)^{\mu_{mj}}\le|u|^{\sum_{m=1}^{I}\mu_{mj}} \text{ and } \prod_{m=1}^{I}(u_m)^{\nu_{mj}}\le|u|^{\sum_{m=1}^{I}\nu_{mj}}.
	\end{align*}
	Eventually, we can deduce that
	\begin{align}\label{eqn:gc}
		|f_i(u)|\le C|u|^{\lambda}, \text{ for all } i=1,2, . . ., I,
	\end{align}
	where $C$ is a constant.
	
	A significant application of this model is the reactive transport of species in the groundwater flow but setting the porosity to unity. We consider the reversible reaction rates of the mass action type. Most of the hydrogeochemical models \cite{fang2003general, friedly1992solute, krautle2007reduction, lichtner1985continuum} contains mass action kinetics and hence plays a crucial role in computational hydrogeochemistry. We do not restrict the scope of the model by imposing any restrictions on the number of mobile species, the number of chemical reactions, the dimension of the space, the value of the stoichoimetric coefficients or the degree of the non-linearity of the reaction rates.

	Our main focus is to establish the existence of global-in-time classical solutions together with the uniform boundedness of the solutions. Although the question of global existence for reaction-diffusion systems is a classical topic but still poses a lot of open and challenging issues. The main difficulty in proving the existence and boundedness of the solutions of the system lies in the finding of a-priori $L^p-$estimates.  We will study the problem $(\mathcal{P})$ under the following assumptions:
	\begin{itemize}
		\item[\bf (A1.)] Let $\Omega\subset\mathbb{R}^n$ is a bounded domain with smooth boundary $\partial\Omega$. Initial data $u_i^0\in L^\infty(\Omega)$ and $u_i^0(x)\ge 0$ together with the diffusion coefficient $D_i>0$ is a constant, for all $i=1, 2, . . ., I$.
		\item[\bf (A2.)] $b_i\in L^p(S\times\partial\Omega)$ and $b_i\le0$, for all $i=1, 2, . . ., I$.
		\item[\bf (A3.)]  $f_i: \mathbb{R}^I\rightarrow \mathbb{R}$ is Locally Lipschitz continuous function for all $i=1, 2, . . ., I$.
		\item[\bf (A4.)] (\textbf{Mass control}) $\sum_{i=1}^{I}f_i(u)\le C_1\sum_{i=1}^{I} u_i+C_2$ for all $u\in \mathbb{R}^I_{+}$, for some constants $C_1, C_2>0$.
	\end{itemize}
	Here we only consider the natural assumption of mass control. Mass control is basically a generalization to the condition of mass conservation
	\begin{align}\label{eqn:mc}
		\sum_{i=1}^{I}f_i(u)=0
	\end{align} 
	and to the condition of mass dissipation
	\begin{align}\label{eqn:md}
		\sum_{i=1}^{I}f_i(u)\le0.
	\end{align}
	The nonnegativity of the initial condition is taken to prove the nonnegativity of the solutions $u_i$. The locally Lipschitz assumption \textbf{(A3.)} gives the local existence of the classical solutions to the system on $[0, T_{\max}),$ see \cite{amann1985global} for details. The local existence results have long been well known and basically they form the basis of the most global existence arguments. The global existence of this local solutions is really demanding for the case of different diffusion coefficients. As for the case of equal diffusion coefficients i.e. $D_i=D$, summing up all the equations \eqref{eqn:M11} for $i=1, 2, . . ., I$ leads to
	\begin{align*}
		\frac{\partial}{\partial t}(\sum_{i=1}^{I}u_i)-D\Delta (\sum_{i=1}^{I}u_i)\le C_1\sum_{i=1}^{I} u_i+C_2.
	\end{align*} 
	Relying on \textbf{(A2.)} and the maximum principle, we get the estimate
	\begin{align*}
	\left\lVert u_i(t)\right\rVert_{L^\infty(\Omega)}\le	\left\lVert \sum_{i=1}^{I}u_i(t)\right\rVert_{L^\infty(\Omega)}\le C e^T.
	\end{align*}
	This gives the global existence of the solutions and application of the bootstrap argument will give the uniform in time boundedness of the solutions. However the situation changes dramatically when the diffusion coefficients are different and \cite{pierre2010global} gives an extensive survey on this topic. The main difficulty is the lack of maximal principle for general systems to yield the $L^\infty$- a-priori estimates.

	\subsection{Literature Survey}
	We wish to put the present work in context by drawing a distinction between what is already available in the literature and what we have done new. To begin with, we first review some existing results in the literature. The work of Pierre and Schmidt \cite{pierre2000blowup} pointed out that only mass control and positivity of the solutions is not sufficient to prevent the blow up and therefore we need a growth control condition. In this aspect we deduce \eqref{eqn:gc} to control the growth of the nonlinearities. Several works have been done in this direction of establishing the weak solutions for the system $(\mathcal{P})$ with homogeneous Neumann boundary condition. \cite{pierre2003weak} proved the existence of weak solutions for the problem  $(\mathcal{P})$ under the assumption that the nonlinearities belongs to $L^1(S\times\Omega)$. For the case of quadratic nonlinearities, existence of weak solutions is proved in \cite{desvillettes2007global} by using the duality method. The global existence of weak solutions for the case of equal diffusion coefficients and inhomogeneous Neumann boundary conditions can be found in \cite{krautle2011existence} for the solutions space $[H^{1, p}(S; L^p(\Omega))\cap L^p(S; H^{2, p}(\Omega))]^I$ and $p>n+1$. The same result for a less regular solutions space $[H^{1, p}(S; H^{1, q}(\Omega)^{*})\cap L^p(S; H^{2, p}(\Omega))]^I$ and $p>n+2$ is done in \cite{mahato2013global}. 
	
	The authors in \cite{fischer2015global, pierre2017asymptotic} proved the existence of global renormalized solutions if the nonlinearities satisfy the entropy condition
	\begin{align}\label{eqn:ei}
		\sum_{i=1}^{I}f_i(u)(\log u_i+\alpha_i)\le 0 \text{ for all } u\in (0, \infty)^I,
	\end{align}
	for some $\alpha_1, \alpha_2, . . ., \alpha_I\in\mathbb{R}$. Although the renormalized solutions does not ensure the $L^1$ integrability of the nonlinear reaction rate terms. However, it coincides with the unique classical solutions as long as the classical solutions exists \cite{fischer2017weak}. 
	
	Now coming to the conditional global existence of classical solutions which is an active research direction nowadays. Global classical solutions for the system ($\mathcal{P}$) with homogeneous Neumann boundary condition  under the assumption \eqref{eqn:mc} and \eqref{eqn:ei} is proved in \cite{goudon2010regularity} if dimension $n=1$ and $f_i(u)$ have cubic growth and for $n=2$, $f_i(u)$ have quadratic growth. The result later improved by \cite{tang2018global} by utilizing a modified Gagliardo-Nirenbarg inequality. Further, the dimension restriction was removed in the work \cite{caputo2009global} and existence of global classical solutions is proved for strictly sub-quadratic growth. We refer to \cite{canizo2014improved} where the global existence under mass conservation for two dimensional domain with quadratic growth is proved and the boundedness in time is shown in \cite{pierre2017asymptotic}. Further, for higher order nonlinearities in any space dimension global existence of classical solutions is proved under the quasi-uniform diffusion condition, i.e. the diffusion coefficients are close enough to each other. Similar work can be found in \cite{fellner2016global, cupps2021uniform}. However, \cite{canizo2014improved} shown that the $L^\infty$-norm of the classical solutions is growing at most polynomially and later on this is removed in \cite{cupps2021uniform}. The authors in \cite{pierre2017dissipative} proved the global existence and uniform boundedness for quadratic growth and for dimension $n=2$ by relaxing the mass conservation assumption \eqref{eqn:mc} to mass dissipation \eqref{eqn:md}. This result is improved in \cite{morgan2020boundedness} by replacing the mass dissipation assumption by a weaker intermediate sum condition. In higher dimensions, the global existence of classical solutions with quadratic nonlinearities has been proved in \cite{caputo2019solutions,fellner2020global,souplet2018global} and for the case of $\Omega=\mathbb{R}^n$ is deduced in \cite{kanel1990solvability}. The work \cite{caputo2019solutions} is done based on mass conservation \eqref{eqn:mc} together with the entropy condition \eqref{eqn:ei}, where as \cite{souplet2018global} relaxed the mass conservation condition to mass dissipation \eqref{eqn:md}. The more general work is done in \cite{fellner2020global} under the mass control assumption. The uniform in time bound for the solutions is shown in \cite{fellner2021uniform}.
	
	We can see in the aforementioned works that the global existence of classical solution in any space dimensions and for the higher order nonlinearities is proved under the restriction that the diffusion coefficients are close to each other. Along with that most of them does not talk about uniform in time bounds of the solution. Therefore the main motivation of this work comes from relaxing the quasi-uniform diffusion condition assumption together with working with a rather weaker natural mass control condition. We also taken care of inhomogeneous Neumann boundary conditions along with a less regular solution space upto the boundary. We basically modified the work \cite{cupps2021uniform} suitably and introduce a new choice of common diffusion coefficients to derive the a-priori estimates. Our work only needs the restriction that the maximal regularity constant should be less than $1$, which we also shown that a feasible assumption. 
	
	The structure of this paper as follows: In Section 2, we collect all mathematical tools and Lemmas require to analyze the system in Section 3 and state the main result of the paper. Section 3 is devoted to the proof of the main theorem. First we start with the positivity of the concentrations then we collect all a-priori estimates and finally prove the existence, uniqueness and global in time bound of the solution. 
	
     \section{Mathematical Preliminaries}
    Let $1<p<\infty$, we define $p^{'}$ as the Holder conjugate exponent of $p$ such that $\frac{1}{p}+\frac{1}{p^{'}}=1$. For $1\le p\le\infty, L^p(\Omega)$ is the space of real-valued measurable functions $u(x)$ such that $|u(.)|^p$ is Lebesgue measurable with the corresponding norm given by 
    \begin{align*}
    	\|u\|_{p,\Omega}=\left(\int_{\Omega}|u(x)|^pdxdt\right)^{\frac{1}{p}}\text{ for } 1<p<\infty
    \end{align*}
   and 
   \begin{align*}
   	\|u\|_{\infty,\Omega}=\underset{x \in \Omega}{\text{ess sup}}{|u(x)|} \text{ for } p=\infty.
   \end{align*}
    Similarly, the Bochner space $L^p(\tau, T; L^p(\Omega))$ is denoted by $L^p((\tau, T)\times\Omega)$ together with the norms
    \begin{align*}
    	\|u\|_{p, (\tau, T)\times\Omega}=\begin{cases}
    		\left(\int_{\tau}^{T}\int_{\Omega}|u(t,x)|^pdxdt\right)^{\frac{1}{p}} \text{ for } 1\le p<\infty\\
    		\underset{(\tau, T)\times\Omega}{\text{ess sup}}{|u(t, x)|} \text{ for } p=\infty.
    	\end{cases}
    \end{align*}
    The space $C(\Omega)$ and $C^k(\bar{\Omega})$ denotes the Banach space of all continuous functions and the Banach space of all $k-$times continuously differentiable functions w.r.t the norms
    \begin{align*}
    	\|u\|_{C(\Omega)}=\underset{x\in\Omega}{\sup} |u(x)|\text{ and } \|u\|_{C^k(\bar{\Omega})}=\sum_{|\alpha|\le k}\underset{x\in \bar{\Omega}}{\sup} |D^\alpha u(x)|.
    \end{align*} 
  Suppose $0<\gamma\le 1$, $C^\gamma(\bar{\Omega})$ consists of all functions $u\in C(\bar{\Omega})$ such that
  \begin{align*}
  	\|u\|_{C^\gamma(\bar{\Omega})}=\|u\|_{C(\bar{\Omega})}+\underset{\underset{x\neq y}{x,y\in\Omega} }{\sup}\left\{\frac{|u(x)-u(y)|}{|x-y|^\gamma}\right\}<\infty.
  \end{align*}
  The space $H^{1,p}(\Omega)$ is the usual sobolev space w.r.t the the norm
  \begin{align*}
  	\|u\|_{H^{1, p}}=\begin{cases}
  		\left(\|u\|^p_{p,\Omega}+\|\nabla u\|^p_{p, \Omega}\right)^{\frac{1}{p}} \text{ for } 1\le p<\infty\\
  		\underset{x\in\Omega}{\text{ess sup}}{|u(x)|+|\nabla u(x)|} \text{ for } p=\infty.
  	\end{cases}
  \end{align*}
  The Banach Space 
  \begin{align*}
  	W_p^{(2, 1)}(S\times\Omega)=\left\{u | u, \frac{\partial u}{\partial t}, \frac{\partial u}{\partial x_i}, \frac{\partial^2 u}{\partial x_i\partial x_j}\in L^p(S\times\Omega)\text{ for all } i,j=1, 2, . . ., n\right\}
  \end{align*}
 is the solution space together with the norm
 \begin{align*}
 	\|u\|^{(2, 1)}_{p, S\times\Omega}=\left(\|u\|^p_{p, S\times\Omega}+\left\lVert\frac{\partial u}{\partial t}\right\rVert^p_{p, S\times\Omega}+\sum_{i=1}^{n}\left\lVert\frac{\partial u}{\partial x_i}\right\rVert^p_{p, S\times\Omega}+\sum_{i, j=1}^{n}\left\lVert\frac{\partial^2 u}{\partial x_i\partial x_j}\right\rVert^p_{p, S\times\Omega}\right)^{\frac{1}{p}}.
 \end{align*}
  $H^{1, q}(\Omega)^*$ denotes the dual of the space $H^{1, q}(\Omega)$. We denote the duality pairing between $H^{1, q}(\Omega)$ and $H^{1, q}(\Omega)^*$ by $\langle u, v\rangle_{H^{1,q}(\Omega)^{*}\times H^{1,q}(\Omega)}$. We define $L^{p}(\Omega)\hookrightarrow H^{1,q}(\Omega)^{*}$ as 
  \begin{eqnarray}
  	\langle f, v\rangle_{H^{1,q}(\Omega)^{*}\times H^{1,q}(\Omega)}&=&\langle u,v\rangle_{L^{p}(\Omega)\times L^{q}(\Omega)}:=\int_{\Omega}u v\,dx \textnormal{ for }u\in L^{p}(\Omega),\; v\in H^{1,q}(\Omega).\notag	
\end{eqnarray}
 Now for a $\beta\in\mathbb{R}$, we define $\beta_{+} := \max\{ \beta, 0 \}\ge 0 \; \text{ and } \beta_{-} := \max\{ -\beta, 0 \}\ge 0$ such that $\beta=\beta_{+}-\beta_{-}$. The set of non negative integers is denoted by $\mathbb{Z}_0^+=\mathbb{Z}^+\cup\{0\}=\{\beta\in\mathbb{Z} | \beta>0 \}\cup\{0\}$.
 \begin{lemma}\label{lemma:ur}
 	For $m=1, 2, . . ., I$ and $a_l, \bar{a}_l\in \mathbb{R}$
 	\begin{align*}
 		a_1a_2. . . a_I-\bar{a}_1\bar{a}_2. . . \bar{a}_I=\sum_{m=1}^{I}a_1a_2. . .a_{m-1}(a_m-\bar{a}_m)\bar{a}_{m+1}\bar{a}_{m+2}. . .\bar{a}_I.
 	\end{align*}
 \end{lemma}
\begin{proof}
	It is mere calculation.
\end{proof}
 \begin{lemma}(Trace Theorem)
 	Let $1\le p<\infty$ and $\Omega\subset\mathbb{R}^n$ be a bounded domain with sufficiently smooth boundary $\partial\Omega$. Then there exists a bounded linear operator $T : H^{1, p}(\Omega) \rightarrow L^p(\partial\Omega)$ such that
 	\begin{align*}
 		(i) Tu := u|_{\partial\Omega}\text{ if } u\in H^{1,p}(\Omega)\cap C(\bar{\Omega})
 	\end{align*}
  and
  \begin{align*}
  	(ii) \|Tu\|_{p, \partial\Omega}\le C\|u\|_{H^{1,p}(\Omega)} \text{ for each } u\in H^{1, p}(\Omega),
  \end{align*}
 where $C$ depends on $p$ and $\Omega$ but it is independent of $u$.
 \end{lemma}
  \begin{proof}
  	The proof is done in the Theorem 1 of Section 5.5 of \cite{evans2010partial}.
  \end{proof}
    \begin{defin}(Classical Solutions \cite{lunardi2012analytic}).\label{dfn:1}\\
    	A set of concentrations $u=(u_1, u_2, . . ., u_I)$ is said to be a classical solution for the system $(\mathcal{P})$ in the interval $S$ if $u_i\in C([0, T); L^p(\Omega))\cap C^1((0, T)\times\bar{\Omega})$, for all $i=1, 2, . . ., I$ and $p>n$ and $u$ satisfies each equation in $(\mathcal{P})$ point-wise.
    \end{defin}
    \begin{lemma}(Maximal Regularity \cite{lamberton1987equations}).\label{lemma:1}\\
    	Let $0<\tau<T$ and $p\in (1,\infty)$. Let $0\le\theta\in L^p((\tau,T)\times\Omega)$ and $\|\theta\|_{p,(\tau,T)\times\Omega}=1$ and $\psi$ be the weak solution to
    	\begin{equation}\label{eqn:MR1}
    		\left\{ \begin{aligned} 
    			\frac{\partial\psi}{\partial t}+\Delta\psi&=-\theta\text{ in }(\tau,T)\times\Omega,\\
    			\nabla\psi.\vec{n}&=0\text{ on } (\tau,T)\times\partial\Omega,\\
    			\psi(T, x)&=0 \text{ in } \Omega.
    		\end{aligned}\right.
    	\end{equation}
    	Then $\psi\ge0$, 
    	\begin{align}\label{eqn:MR2}
    		\|\psi\|^{2,1}_{p, (\tau, T)\times\Omega}\le C_{T-\tau,p}
    	\end{align}
    	and 
    	\begin{align}\label{eqn:MR3}
    		\|\Delta\psi\|_{p, (\tau,T)\times\Omega}\le C_{mr}(p),
    	\end{align}
     where $C_{mr}(p)$ in the maximal regularity constant independent of $\tau$ and $T$ but depends on $p$.
    \end{lemma}
    \begin{lemma}(Embedding Inequalities \cite{ladyvzenskaja1988linear}).\label{lemma:2}
    	Let $1<p<\infty$.\\
    	$(i)$ If $p\le \frac{n+2}{2}$, then for all $f\in W^{2,1}_p((\tau, T)\times\Omega)$ we have
    	\begin{align*}
    		\|f\|_{q, (\tau,T)\times\Omega}\le C(p,T-\tau)\|f\|^{2,1}_{p, (\tau, T)\times\Omega},\text{ for all } 1\le q<\frac{(n+2)p}{n+2-2p}.
    	\end{align*}
    	$(ii)$ If $p>\frac{n+2}{2}$, then 
    	\begin{align*}
    		\|f\|_{\infty, (\tau, T)\times\Omega}\le C(p, T-\tau)\|f\|^{2,1}_{p, (\tau, T)\times\Omega},
    	\end{align*}
    	where the constant $C(p, T-\tau)$ depends on $\Omega, p$ and $(T-\tau)$.
    \end{lemma}
   We are now ready to state the main result of this paper:
    \begin{thm}\label{thm:1}
    	Let the assumptions $({\bf A1.})- ({\bf A4.})$ holds true. If the maximal regularity constant $C_{mr}(p^{'})<1$ for some $p^{'}$ as explained in Lemma \ref{lemma:1}, where $\frac{1}{p}+\frac{1}{p^{'}}=1$. Then there exists a unique positive global in time classical solution $u_i\in W^{2, 1}_p(S\times\Omega)$, for all $i=1, 2, . . ., I$ of the problem $(\mathcal{P})$. Moreover, the solution $u_i\in C(S ; L^p(\Omega))\cap C^1(S\times\bar{\Omega})$. In other words, the solution exists in the classical sense defined in the Definition \ref{dfn:1}. Furthermore, the solution in bounded uniformly in time, i.e.
    	\begin{align*}
    		\sup_{t\ge0}\|u_i(t)\|_{\infty, \Omega}<+\infty \text{ for all } i=1, 2, . . .,I.
    	\end{align*}
    \end{thm}
    \begin{cor}
    	We can establish that there always exists some $p^{'}\in[\frac{3}{2}, 2]$ such that $C_{mr}(p^{'})<1$.\\
    	We have from Lemma 3.2 of \cite{canizo2014improved} that for any $mr>0$ and $\frac{3}{2}\le r \le2$ the inequality holds:
    	\begin{align}\label{eqn:ci}
    		C_{mr}(r)\le (mr)^{-\frac{4}{r}(r-\frac{3}{2})}(C_{mr}(\frac{3}{2}))^{\frac{3}{r}(2-r)}.
    	\end{align}
    	Now we choose $r=p^{'}$ and $mr=\frac{D_{\max}}{2}$. Then $C_{mr}(p^{'})<1$ is satisfied if the estimate is satisfied:
    	\begin{align}\label{eqn:d1}
    		-\frac{4}{p^{'}}(p^{'}-\frac{3}{2})\ln(\frac{D_{\max}}{2})+\frac{3}{p^{'}}(2-p^{'})\ln(C_{\frac{D_{\max}}{2}}(\frac{3}{2}))<0.
    	\end{align}
    	We put $g(p^{'})=-\frac{4}{p^{'}}(p^{'}-\frac{3}{2})$, then  $1-g(p^{'})=\frac{3}{p^{'}}(2-p^{'})$. Thus \eqref{eqn:d1} yields
    	\begin{align}
    		&(1-g(p^{'}))\left(\ln(\frac{D_{\max}}{2})+\ln(C_{\frac{D_{\max}}{2}}(\frac{3}{2}))\right)<\ln(\frac{D_{\max}}{2})\notag\\
    		\implies& \frac{3}{p^{'}}(2-p^{'})\ln\left(\frac{D_{\max}}{2}C_{\frac{D_{\max}}{2}}(\frac{3}{2})\right)<\ln(\frac{D_{\max}}{2}).\label{eqn:de}
    	\end{align}
    	Therefore for $\frac{D_{\max}}{2}C_{\frac{D_{\max}}{2}}(\frac{3}{2})>1$, the inequality \eqref{eqn:de} holds true as soon as 
    	\begin{align*}
    		2-p^{'}<\frac{p^{'}}{3}\frac{\ln(\frac{D_{\max}}{2})}{\ln\left(\frac{D_{\max}}{2}C_{\frac{D_{\max}}{2}}(\frac{3}{2})\right)}.
    	\end{align*}
    	That means we have to choose a $p^{'}\in [\frac{3}{2}, 2]$ satisfying
    	\begin{align*}
    		2-p^{'}<\frac{1}{2}\frac{\ln(\frac{D_{\max}}{2})}{\ln\left(\frac{D_{\max}}{2}C_{\frac{D_{\max}}{2}}(\frac{3}{2})\right)}.
    	\end{align*}
    	Similarly, for $\frac{D_{\max}}{2}C_{\frac{D_{\max}}{2}}(\frac{3}{2})<1$ we need to pick a $p^{'}\in [\frac{3}{2}, 2]$ satisfying
    	\begin{align*}
    		2-p^{'}>\frac{2}{3}\frac{\ln(\frac{D_{\max}}{2})}{\ln\left(\frac{D_{\max}}{2}C_{\frac{D_{\max}}{2}}(\frac{3}{2})\right)},
    	\end{align*}
    	so that \eqref{eqn:de} is satisfied.
    \end{cor}

    \section{Existence of Global Classical Solution}
    \subsection{Positivity of the concentrations}
    We utilize the idea of \cite{krautle2011existence} to establish the non-negativity of the concentrations.  We consider the system
    \begin{subequations}
    	\begin{align}
    		\frac{\partial u_i}{\partial {t}} + \nabla.(-{D_i}\nabla u_i)&=f_i([u]_+) \; \text{ in } S\times \Omega,	\label{eqn:M1p}\\
    		-{D_i}\nabla u_i.\vec{n}&=b_i \; \text{ on } S \times \partial\Omega, \label{eqn:M2p}\\
    		u_i(0,x)&=u_i^0(x) \; \text{ in } \Omega,\label{eqn:M3p} 
    	\end{align}	
    \end{subequations}
   We denote the above system by $(\mathcal{P_+})$.
    \begin{lemma}
    	Under the assumptions $({\bf A1.})- ({\bf A4.})$ let us assume that $u_i$ is the concentration satisfying the system $(\mathcal{P_+})$. Then $u_i\ge0$ holds for all $i=1, 2, . . ., I$. Hence $u_i$ satisfies the system $(\mathcal{P})$.
    \end{lemma}
     \begin{proof}
     	We multiply the equation \eqref{eqn:M1p} by $-\chi_{(0,t)}[u_i]_-$ and integrating over $S\times\Omega$ obtain
     	\begin{align}
     		\frac{1}{2}\|[u_i(t)]_-\|^2_{2, \Omega}&+D_i\|\nabla [u_i]_{-}\|^2_{2, S\times\Omega}=\frac{1}{2}\|[u_i^0]_-\|^2_{2, \Omega}\nonumber+\int_{0}^{t}\int_{\partial\Omega}b_i[u_i]_-d\sigma_xdt\\
     		&+\sum_{j=1}^{J}(\nu_{ij}-\mu_{ij})\int_{0}^{t}\int_{\Omega} (R_j^f([u]_+)-R_j^b([u]_+))(-[u_i]_-)dxdt.
     	\end{align}
     Relying on $({\bf A1.})$ we can see that, $[u_i^0]_-=0$. By $({\bf A2.})$, we can see that $\int_{0}^{t}\int_{\partial\Omega}b_i[u_i]_-d\sigma_xdt\le0$.  If $\nu_{ij}>0$ then $R_j^b([u]_+)$ contains a non-trivial factor and $[u]_+[u]_-=0$. Similarly, If $\mu_{ij}>0$ then $R_j^f([u]_+)$ contains a non-trivial factor and $[u]_+[u]_-=0$. We finally deduce that
     \begin{align*}
     	\frac{1}{2}\|[u_i(t)]_-\|^2_{2,\Omega}+D_i\|\nabla [u_i]_{-}\|^2_{2,S\times\Omega}\le \int_{0}^{T}\int_{\Omega} \sum_{j=1}^{J}(\nu_{ij} R_j^f([u]_{+})+\mu_{ij}R_j^b([u]_+)))(-[u_i]_-)dxdt\le 0.
     \end{align*}
   This gives $[u_i(t)]_-=0$. Therefore $u_i(t)\ge 0$.
     \end{proof}
    \subsection{Finding the a-priori estimates}
    \begin{lemma}\label{lemma:3}
    	Let us set 
    	\begin{align}\label{eqn:dc}
    		 D=\frac{D_{\max}}{2}, \text{ where }D_{\max}=\max_{i=1, 2, . . ., I} D_i.
    	\end{align}
    	 Now let $p\in(1,\infty)$ and $\theta\in L^p((\tau, T)\times\Omega)$ and $\psi$ be the solution of 
    	\begin{align}
    		\begin{split}\label{eqn:pe}
    			\frac{\partial\psi}{\partial t}+D\Delta\psi&=-\theta\text{ in }(\tau,T)\times\Omega,\\
    			\nabla\psi.\vec{n}&=0\text{ on } (\tau,T)\times\partial\Omega,\\
    			\psi(T,x)&=0 \text{ in } \Omega.
    		\end{split}
    	\end{align}
      Then 
      \begin{itemize}
      	\item[$(i)$] $\|\Delta\psi\|_{p,(\tau,T)\times\Omega}\le \frac{C_{mr}(p)}{D}\|\theta\|_{p,(\tau, T)\times\Omega}.$
      	\item[$(ii)$] $\left\lVert\frac{\partial \psi}{\partial t}\right\rVert_{p,(\tau,T)\times\Omega}\le (C_{mr}(p)+1)\|\theta\|_{p, (\tau,T)\times\Omega}$.
      \end{itemize}
     Moreover, there exists a constant $C_{T-\tau, p}$ depending on $p$ and $(T-\tau)$ but independent of $D$ such that
     \begin{align}
     	\|\psi\|^{(2,1)}_{p,(\tau,T)\times\Omega}&\le
     	\begin{cases}
     		 C_{T-\tau,p}\|\theta\|_{p, (\tau, T)\times\Omega}\text{ when } 0<D<1\\
     		 C_{T-\tau,p}\|\theta\|_{p, (\tau, T)\times\Omega}(1+C_{mr}(p))\text{ when } D>1.
     	\end{cases}
     \end{align}
    
    \end{lemma}
   \begin{proof}
   	Let $z(t,x)=\psi(\frac{t}{D}, x)$ then $z$ satisfies
   	\begin{align*}
   		\frac{\partial z}{\partial t}+\Delta z&=-\frac{\tilde{\theta}}{D}
   		\text{ in }(D\tau, DT)\times\Omega,\\
   		\nabla\psi.\vec{n}&=0\text{ on } (D\tau, DT)\times\partial\Omega,\\
   		\psi(DT,x)&=0 \text{ in } \Omega,
   	\end{align*}
    where $\tilde{\theta}(t,x)=\theta(\frac{t}{D}, x)$. Then by Lemma \ref{lemma:1} we have
    \begin{align*}
    	\|\Delta z\|_{p, (D\tau, DT)\times\Omega}\le C_{mr}(p)\left\lVert\frac{\tilde{\theta}}{D}\right\rVert_{p,(D\tau, DT)\times\Omega}.
    \end{align*}
  We substitute $s=\frac{t}{D}$ and calculate to get
  \begin{align*}
  	\int_{\tau}^{T}\int_{\Omega}|\Delta\psi|^pdxds&\le\left(\frac{C_{mr}(p)}{D}\right)^p\int_{\tau}^{T}\int_{\Omega}|\theta|^pdxds\\
  	\implies \|\Delta\psi\|_{p,(\tau,T)\times\Omega}&\le\frac{C_{mr}(p)}{D}\|\theta\|_{p,(\tau,T)\times\Omega}.
  \end{align*}
  We can obtain
  \begin{align*}
  	\left\lVert\frac{\partial \psi}{\partial t}\right\rVert_{p,(\tau,T)\times\Omega}\le\|-\theta-D\Delta\psi\|_{p,(\tau,T)\times\Omega}\le (C_{mr}(p)+1)\|\theta\|_{p,(\tau,T)\times\Omega}.
  \end{align*}
  Finally, we can see that
  \begin{align*}
  	\frac{\partial\psi}{\partial t}+\Delta\psi=-\theta-D\Delta\psi+\Delta\psi=-[\theta+(D-1)\Delta\psi].
  \end{align*}
  Then Lemma \ref{lemma:1} implies that
  \begin{align*}
  	\|\psi\|^{(2, 1)}_{p,(\tau,T)\times\Omega} &\le C_{T-\tau, p}\|\theta+(D-1)\Delta\psi\|_{p,(\tau,T)\times\Omega}\\
  	&\le C_{T-\tau, p}[1+\left(1-\frac{1}{D}\right)C_{mr}(p)]\|\theta\|_{p,(\tau,T)\times\Omega}\\
  	&\le\begin{cases}
  		C_{T-\tau,p}\|\theta\|_{p, (\tau, T)\times\Omega}\text{ when } 0<D<1\\
  		C_{T-\tau,p}\|\theta\|_{p, (\tau, T)\times\Omega}(1+C_{mr}(p))\text{ when } D>1.
  	\end{cases}.
  \end{align*}
   \end{proof}
  \begin{prop}\label{prop:1}
  	Suppose the assumptions $\bf (A1.)-\bf (A4.)$ hold. Set 
  	\begin{align}
  		D_{\max}=\max_{i=1, 2, . . ., I} D_i \text{ and } D=\frac{D_{\max}}{2}.
  	  	\end{align}
    If 
    \begin{align}
    	C_{mr}(p^{'})<1 \text{ for some } p>\frac{(\lambda-1)(n+2)}{2},
    \end{align}
 where $\frac{1}{p}+\frac{1}{p^{'}}=1$ and $\lambda$ is defined as \eqref{eqn:gc}. Then we have
   \begin{align}
   	\|u_i(t)\|_{p, S\times \Omega}\le C(T).
   \end{align}
  \end{prop}
 \begin{proof}
 	Let $\theta\in L^{p^{'}}(S\times\Omega)$ such that $\|\theta\|_{p^{'}, S\times\Omega}=1$. Then by Lemma \ref{lemma:3} we have
 	\begin{align*}
 		\int_{0}^{T}\int_{\Omega}u_i\theta dxdt&=\int_{0}^{T}\int_{\Omega} u_i(-\frac{\partial \psi}{\partial t}-D\Delta \psi)dxdt=\int_{\Omega} u_i^0(x)\psi(0,x) dx+\int_{0}^{T}\int_{\Omega} (D_i-D)u_i\Delta\psi dxdt\\
 		&+\int_{0}^{T}\int_{\Omega}f_i(u)\psi dxdt-\int_{0}^{T}\int_{\partial\Omega}b_i\psi d\sigma_xdt.
 	\end{align*}
 	As $(D_i-D)\le D_{\max}-\frac{D_{\max}}{2}=\frac{D_{\max}}{2}$ and taking summation over $i=1, 2, . . ., n$ we see that
 	\begin{align*}
 		\sum_{i=1}^{I}\int_{0}^{T}\int_{\Omega} u_i\theta dxdt&\le \sum_{i=1}^{I}\int_{0}^{T}\int_{\Omega} u_i^0(x)\psi(0,x)dx+\frac{D_{\max}}{2}\int_{0}^{T}\int_{\Omega}\left (\sum_{i=1}^{I}u_i\right)\Delta\psi dxdt\\
 		&+\int_{0}^{T}\int_{\Omega}\left(\sum_{i=1}^{I}f_i(u)\right)\psi dxdt+\int_{0}^{T}\int_{\partial\Omega}\left(\sum_{i=1}^{I}b_i\right)\psi d\sigma_xdt.
 	\end{align*}
 	Using $(\bf A4.)$ we can simplify it as
 	\begin{align}\label{eqn:1}
 		\sum_{i=1}^{I}\int_{0}^{T}\int_{\Omega} u_i\theta dxdt&\le \sum_{i=1}^{I}\int_{0}^{T}\int_{\Omega} u_i^0(x)\psi(0,x)dx+\frac{D_{\max}}{2}\int_{0}^{T}\int_{\Omega}\left(\sum_{i=1}^{I}u_i\right)\Delta\psi dxdt\notag\\
 		&+C_1\int_{0}^{T}\int_{\Omega}\left(\sum_{i=1}^{I}u_i\right)\psi dxdt+C_2\int_{0}^{T}\int_{\Omega}\psi dxdt+\int_{0}^{T}\int_{\partial\Omega}\left(\sum_{i=1}^{I}b_i\right)\psi d\sigma_xdt.
 	\end{align}
 	By the help of Holder's inequality we see that
 	\begin{align*}
 		\sum_{i=1}^{I}\int_{0}^{T}\int_{\Omega} u_i\theta dxdt&\le 	\sum_{i=1}^{I}\|u_i^0(x)\|_{p,\Omega}\left\lVert\psi(0, .)\right\rVert_{p^{'}, \Omega}+\frac{D_{\max}}{2}\left\lVert\sum_{i=1}^{I}u_i\right\rVert_{p, S\times\Omega}\|\Delta\psi\|_{p^{'}, S\times\Omega}\\
 		&+C_1\left\lVert\sum_{i=1}^{I}u_i\right\rVert_{p, S\times\Omega}\|\psi\|_{p^{'}, S\times\Omega}+C_2(T|\Omega|)^{\frac{1}{p}}\|\psi\|_{p^{'}, S\times\Omega}+\left\lVert\sum_{i=1}^{I}b_i\right\rVert_{p, S\times\partial\Omega}\left\lVert\psi\right\rVert_{p^{'}, S\times\partial\Omega}.
 	\end{align*}
 	Now we can calculate
 	\begin{align*}
 		\|\psi(0, .)\|^{p^{'}}_{p^{'}, \Omega}&=\int_{\Omega}\left\lvert\int_{0}^{T}\frac{\partial \psi}{\partial t}dt\right\rvert^{p^{'}}dx \text{ as } \psi(T,x)=0 \\
 		&\le T^{(p^{'}-1)}(C_{mr}(p^{'})+1))^{p^{'}}.
 	\end{align*}
 	Applying Lemma \ref{lemma:3} we obtain
 	\begin{align*}
 		\sum_{i=1}^{I}\int_{0}^{T}\int_{\Omega} &u_i\theta dxdt\le T^{\frac{1}{p}}(1+C_{mr}(p^{'}))\sum_{i=1}^{I}\|u_i^0\|_{p,\Omega}+ C_{mr}(p^{'})\left\lVert\sum_{i=1}^{I}u_i\right\rVert_{p, S\times\Omega}\\
 		&+C_1\left\lVert\sum_{i=1}^{I}u_i\right\rVert_{p, S\times\Omega}\|\psi\|^{(2,1)}_{p^{'},(\tau,T)\times\Omega}+C_2(T\Omega)^{\frac{1}{p}}\|\psi\|^{(2,1)}_{p^{'},(\tau,T)\times\Omega}+\left\lVert\sum_{i=1}^{I}b_i\right\rVert_{p, S\times\partial\Omega}\|\psi\|^{(2,1)}_{p^{'},(\tau,T)\times\partial\Omega}.
 	\end{align*}
 	We then use duality and get
 	\begin{align*}
 		\left\lVert\sum_{i=1}^{I}u_i\right\rVert_{p, S\times\Omega}&\le T^{\frac{1}{p}}(1+C_{mr}(p^{'}))\sum_{i=1}^{I}\|u_i^0\|_{p,\Omega}+C_{mr}(p^{'})\left\lVert\sum_{i=1}^{I}u_i\right\rVert_{p, S\times\Omega}\\
 		&+C_1\left\lVert\sum_{i=1}^{I}u_i\right\rVert_{p, S\times\Omega}\|\psi\|^{(2,1)}_{p^{'},(\tau,T)\times\Omega}+C_2(T\Omega)^{\frac{1}{p}}\|\psi\|^{(2,1)}_{p^{'},(\tau,T)\times\Omega}+\left\lVert\sum_{i=1}^{I}b_i\right\rVert_{p, S\times\partial\Omega}\|\psi\|^{(2,1)}_{p^{'},(\tau,T)\times\partial\Omega}.
 	\end{align*}
 	Now we have to consider the two cases one for $0<D<1$ and another one $D>1$. For the case $0<D<1$ we have
 	\begin{align*}
 		(1-C_{mr}(p^{'})-C_1C_{T-\tau, p^{'}})\left\lVert\sum_{i=1}^{I}u_i\right\rVert_{p, S\times\Omega}\le C
 	\end{align*}
 	and for $D>1$ we have 
 	\begin{align*}
 		(1-C_1&C_{T-\tau,p^{'}}-C_{mr}(p^{'})(1+C_1C_{T-\tau,p^{'}}))\left\lVert\sum_{i=1}^{I}u_i\right\rVert_{p, S\times\Omega}\le C.
 	\end{align*}
 	For the first case the coefficient of $\left\lVert\sum_{i=1}^{I}u_i\right\rVert_{p, S\times\Omega}$ will be positive if $C_{mr}(p^{'})<1-C_1C_{T-\tau,p^{'}}<1$ and for the second case $C_{mr}(p^{'})< \frac{(1-C_1C_{T-\tau,p^{'}})}{1+C_1C_{T-\tau,p^{'}}}<1$. Combining these two cases we can conclude that for $C_{mr}(p^{'})<1$ we have
 	\begin{align}\label{eqn:2}
 		\left\lVert\sum_{i=1}^{I}u_i\right\rVert_{p, S\times\Omega}\le C.
 	\end{align}
 	Non-negativity of $u_i$ gives that
 	\begin{align*}
 		\|u_i\|_{p, S\times\Omega}\le C  \text{ for all } i=1, 2, . . ., I.
 	\end{align*}
  \end{proof}
   \subsection{Proof of Theorem \ref{thm:1} }
   We now can prove our main theorem by employing the Lemmas and the Proposition \ref{prop:1}. We can write from \eqref{eqn:M11} and \eqref{eqn:gc} that
    \begin{align*}
    	\frac{\partial u_i}{\partial t}-D_i\Delta u_i=f_i(u)\in L^{\frac{p}{\lambda}}(S\times\Omega).
    \end{align*}
    Then by Maximal Regularity i.e. by Lemma \ref{lemma:1} we have
    \begin{align*}
    	u_i\in W^{2, 1}_{\frac{p}{\lambda}}(S\times\Omega).
    \end{align*}
    Now by Embedding Inequality Lemma \ref{lemma:2} we can see that if $\frac{p}{\lambda}> \frac{(n+2)}{2}$ then 
    \begin{align*}
    	\|u_i\|_{\infty, S\times\Omega}\le C(\frac{p}{\lambda}, T)\|u_i\|^{(2,1)}_{p, S\times\Omega}.
    \end{align*}
    Otherwise, that is when $\frac{p}{\lambda} \le\frac{(n+2)}{2}$ then
    \begin{align*}
    	u_i\in L^{p_1}(S\times\Omega) \text{ where } p_1=\frac{(n+2)\frac{p}{\lambda}}{n+2-\frac{2p}{\lambda}}.
    \end{align*}
    Therefore, we obtain
    \begin{align*}
    	\frac{\partial u_i}{\partial t}-D_i\Delta u_i=f_i(u)\in L^{\frac{p_1}{\lambda}}(S\times\Omega).
    \end{align*}
    We repeat the same argument and deduce that
    \begin{align*}
    	u_i\in L^{p_2}(S\times\Omega) \text{ where } p_2=\frac{(n+2)\frac{p_1}{\lambda}}{n+2-\frac{2p_1}{\lambda}}.
    \end{align*}
    So basically we are having a recursive sequence $\{p_k\}$ such that $u_i\in  L^{p_{k+1}}(S\times\Omega) $ and 
    \begin{align*}
    	p_{k+1}=\frac{(n+2)\frac{p_k}{\lambda}}{n+2-\frac{2p_k}{\lambda}}.
    \end{align*} 
    Now as $p>\frac{(\lambda-1)(n+2)}{2}$, we have
    \begin{align*}
    	p_{k+1}>\left(\frac{n+2}{\lambda(n+2)-2p_0}\right)^kp_0.
    \end{align*}
    Therefore as $k\rightarrow\infty$, $p_k\rightarrow\infty$ that means there exists some $k_0$ for which
    \begin{align*}
    	\frac{p_{k_0}}{\lambda}>\frac{(n+2)}{2}
    \end{align*}
    and hence we finally get
    \begin{align*}
    	\frac{\partial u_i}{\partial t}-D_i\Delta u_i=f_i(u)\in L^{\frac{p_{k_0}}{\lambda}}(S\times\Omega) \text{ and } u_i\in L^\infty(S\times\Omega).
    \end{align*}
    \textbf{Uniqueness}\\
    Let $u^{1}=(u_1^{(1)}, u_2^{(1)}, . . ., u_I^{(1)}), u^{2}=(u_1^{(2)}, u_2^{(2)}, . . ., u_I^{(2)})\in (W_p^{(2,1)})^I$ be two solutions of the problem $(\mathcal{P})$, where $u_i^{(1)}\neq u_i^{(2)}$ for each $i=1, 2, . . ., I$. We denote $\tilde{u}=(\tilde{u}_1,  \tilde{u}_2, . . ., \tilde{u}_I)$ where $(\tilde{u}_i)_{1\le i\le I}=(u_i^{(1)}-u_i^{(2)})_{1\le i\le I}$ then we write the system of equations $\eqref{eqn:M11}-\eqref{eqn:M13}$ for $u_i^{(1)}$ and $u_i^{(2)}$ and take the difference and obtain
    \begin{subequations}
    	\begin{align}
    		\frac{\partial \tilde{u}_i}{\partial {t}} + \nabla.(-{D_i}\nabla \tilde{u}_i)&=f_i(u^1)-f_i(u^2) \; \text{ in } S\times \Omega,	\label{eqn:M1d}\\
    		-{D_i}\nabla \tilde{u}_i.\vec{n}&=0 \; \text{ on } S \times \partial\Omega, \label{eqn:M2d}\\
    		\tilde{u}_i(0,x)&=0 \; \text{ in } \Omega.\label{eqn:M3d} 
    	\end{align}	
    \end{subequations}
    We then test the above equation with $\chi_{(0,t)}\tilde{u}_i$ and deduce
    \begin{align*}
    	\frac{1}{2}\|\tilde{u}_i(t)\|^2_{2,\Omega}+D_i\|\nabla\tilde{u}_i\|^2_{2,S\times\Omega}=\int_{0}^{t}\int_{\Omega}(f_i(u^1)-f_i(u^2))\tilde{u}_i dxdt.
    \end{align*}
    We have to estimate the reaction rate term in the r.h.s. To do so, we first expand the difference $(f_i(u^1)-f_i(u^2))$. Then by Lemma \ref{lemma:ur}, we see that each term contains a factor of the type $(u_i^1-u_i^2)$, which can be estimated by mean value theorem as all other factors are bounded in $L^\infty(S\times\Omega)$. We simplify and finally get
    \begin{align*}
    	\|\tilde{u}_i(t)\|^2_{2,\Omega}\le C\sum_{i=1}^{I}\int_{0}^{t}\|\tilde{u}_i(\tau)\|^2_{2,\Omega}d\tau.
    \end{align*}
    Adding up over $i=1, 2, . . . ,I$ and the application of the Gronwall's inequality yields
    \begin{align*}
    	\|\tilde{u}_i(t)\|^2_{2,\Omega}=0
    	\implies u_i^{(1)}(t)=u_i^{(2)}(t)  \text{ for a.e. } t\in S \text{ and for each } i=1, 2, . . ., I.
    \end{align*}
   \textbf{Uniform in time bounds}\\
  We establish the existence and uniqueness of the global classical solution for the system $\eqref{eqn:M11}-\eqref{eqn:M13}$. Next we wish to prove that the $L^\infty$ norm of the solution is bounded uniformly in time. We define an increasing smooth function $\zeta : \mathbb{R}\rightarrow [0,1]$ such that
  \begin{align*}
  	\zeta(t)=\begin{cases}
  		0 \text{ if } t\le0\\ 
  		1 \text{ if } t\ge 1
  	\end{cases}
  \text{ and } |\zeta'(t)|\le 2 \text{ for all }t\in\mathbb{R}.
  \end{align*}
  For any $\tau\ge0$, we specify the shifted function as $\zeta_\tau(t)=\zeta(t-\tau)$. We then multiply the equation \eqref{eqn:M11} by $\zeta_\tau$ and have the equation
  \begin{align}\label{eqn:uz}
  	\frac{\partial (u_i\zeta_\tau)}{\partial t}-D_i\Delta (u_i\zeta_\tau)=f_i(u)\zeta_\tau+\zeta_\tau^{'} u_i, \text{ for all } \tau\ge0.
  \end{align}
  We choose $\theta\in L^{p^{'}}((\tau, \tau+2)\times\Omega)$ such that $\|\theta\|_{p^{'}, (\tau, \tau+2)\times\Omega}=1$ and $\psi$ be a solution of the system \eqref{eqn:pe} for $T=\tau+2$. We now proceed like \eqref{eqn:1} and deduce that
  \begin{align}
  	&\sum_{i=1}^{I}\int_{\tau}^{\tau+2}\int_{\Omega} u_i\zeta_\tau\theta dxdt\le C_1 \int_{\tau}^{\tau+2}\int_{\Omega} (\sum_{i=1}^{I} u_i\zeta_\tau)\psi dxdt+C_2 \int_{\tau}^{\tau+2}\int_{\Omega}\zeta_\tau\psi dxdt\notag\\
  	&+\sum_{i=1}^{I}\int_{\tau}^{\tau+2}\int_{\Omega} \zeta_\tau^{'}u_i\psi dxdt+\sum_{i=1}^{I}\int_{\tau}^{\tau+2}\int_{\partial\Omega} b_i\zeta_\tau\psi d\sigma_xdt+\frac{D_{\max}}{2}\sum_{i=1}^{I}\int_{\tau}^{\tau+2}\int_{\Omega} u_i\zeta_\tau\Delta\psi dxdt.
  \end{align}
  We now apply Holder's inequality and the fact $|\zeta_\tau'|\le 2$ to get
  \begin{align}\label{eqn:upb}
  	\sum_{i=1}^{I}\int_{\tau}^{\tau+2}\int_{\Omega} u_i\zeta_\tau\theta dxdt&\le C_1\left\Vert \sum_{i=1}^{I} u_i\zeta_\tau\right\rVert_{p, (\tau, \tau+2)\times\Omega}\|\psi\|_{p^{'}, (\tau, \tau+2)\times\Omega}+C_2(2|\Omega|)^{\frac{1}{p}}\|\psi\|_{p^{'}, (\tau, \tau+2)\times\Omega}\notag\\
  	&+2\left\Vert \sum_{i=1}^{I} u_i\right\rVert_{p, (\tau, \tau+2)\times\Omega}\|\psi\|_{p^{'}, (\tau, \tau+2)\times\Omega}+\left\Vert \sum_{i=1}^{I} b_i\right\rVert_{p, (\tau, \tau+2)\times\partial\Omega}\|\psi\|_{p^{'}, (\tau, \tau+2)\times\partial\Omega}\notag\\
  	&+\frac{D_{\max}}{2}\left\Vert \sum_{i=1}^{I} u_i\zeta_\tau\right\rVert_{p, (\tau, \tau+2)\times\Omega}\|\Delta\psi\|_{p^{'}, (\tau, \tau+2)\times\Omega}.
  \end{align}
  Now Lemma \eqref{lemma:3} implies
  \begin{align}\label{eqn:pb}
  	\|\psi\|_{p^{'}, (\tau, \tau+2)\times\Omega} \le \|\psi\|^{(2, 1)}_{p^{'}, (\tau, \tau+2)\times\Omega}\le C_{2, p^{'}}(1+C_{mr}(p^{'}))
  \end{align}
 and 
 \begin{align}\label{eqn:nb}
 	\|\Delta \psi\|_{p^{'}, (\tau, \tau+2)\times\Omega} \le \frac{C_{mr}(p^{'})}{D}=\frac{2 C_{mr}(p^{'}) }{D_{\max}}.
 \end{align}
 We then insert \eqref{eqn:pb} and \eqref{eqn:nb} in \eqref{eqn:upb} and derive
 \begin{align}
 	&\sum_{i=1}^{I}\int_{\tau}^{\tau+2}\int_{\Omega} u_i\zeta_\tau\theta dxdt\le C_1C_{2, p^{'}}(1+C_{mr}(p^{'}))\left\Vert \sum_{i=1}^{I} u_i\zeta_\tau\right\rVert_{p, (\tau, \tau+2)\times\Omega}+C_2(2|\Omega|)^{\frac{1}{p}}C_{2, p^{'}}(1+C_{mr}(p^{'}))\notag\\
 	&+2C_{2, p^{'}}(1+C_{mr}(p^{'}))\left\Vert \sum_{i=1}^{I} u_i\right\rVert_{p, (\tau, \tau+2)\times\Omega}+C C_{2, p^{'}}(1+C_{mr}(p^{'}))+C_{mr}(p^{'})\left\Vert \sum_{i=1}^{I} u_i\zeta_\tau\right\rVert_{p, (\tau, \tau+2)\times\Omega}.
 \end{align}
 We now use duality to obtain
 \begin{align*}
 	(1-C_1C_{2, p^{'}}(1+C_{mr}(p^{'}))-C_{mr}(p^{'}))\left\Vert \sum_{i=1}^{I} u_i\zeta_\tau\right\rVert_{p, (\tau, \tau+2)\times\Omega}\le C_3+C_4 \left\Vert \sum_{i=1}^{I} u_i\right\rVert_{p, (\tau, \tau+2)\times\Omega},
 \end{align*}
 where the constants $C_3$ and $C_4$ are depends on $p$ but independent of the time $T$. Now for $(1-C_1C_{2, p^{'}}(1+C_{mr}(p^{'}))-C_{mr}(p^{'}))>0$, that is $C_{mr}(p^{'})<1$, we have
 \begin{align}
 	\left\Vert \sum_{i=1}^{I} u_i\zeta_\tau\right\rVert_{p, (\tau, \tau+2)\times\Omega}\le C_5+C_6 \left\Vert \sum_{i=1}^{I} u_i\right\rVert_{p, (\tau, \tau+2)\times\Omega}.
 \end{align}
 As $\zeta_\tau\ge 0$ and $\zeta_\tau|_{(\tau+1, \tau+2)}=1$, we are led to
 \begin{align}\label{eqn:3}
 	\left\Vert \sum_{i=1}^{I} u_i\right\rVert_{p, (\tau+1, \tau+2)\times\Omega}\le C_5+C_6 \left\Vert \sum_{i=1}^{I} u_i\right\rVert_{p, (\tau, \tau+2)\times\Omega}.
 \end{align}
  Let $\Im=\left\{\tau\in\mathbb{N} : \left\Vert \sum_{i=1}^{I} u_i\right\rVert_{p, (\tau-1, \tau)\times\Omega}\le \left\Vert \sum_{i=1}^{I} u_i\right\rVert_{p, (\tau, \tau+1)\times\Omega}\right\}$. Then for $\tau+1 \in \Im$, we can rewrite the equation \eqref{eqn:3} as
  \begin{align}
  	\left\Vert \sum_{i=1}^{I} u_i\right\rVert_{p, (\tau+1, \tau+2)\times\Omega}\le C_5+C_6 \left\Vert \sum_{i=1}^{I} u_i\right\rVert_{p, (\tau+1, \tau+2)\times\Omega}.
  \end{align}
  Therefore for $C_6<1$ we obtain
  \begin{align}\label{eqn:4}
  	\left\Vert \sum_{i=1}^{I} u_i\right\rVert_{p, (\tau+1, \tau+2)\times\Omega}\le C \text{ for all } \tau+1\in\Im.
  \end{align}
 Here the constant $C$ is independent of $\tau$. Now if follows from the definition of $\Im$ and the inequality \eqref{eqn:4} that
 \begin{align}\label{eqn:5}
 	\left\Vert \sum_{i=1}^{I} u_i\right\rVert_{p, (\tau+1, \tau+2)\times\Omega}\le C,
 \end{align}
 holds true for all $\tau\in\Im$ and for some constant $C>0$ independent of $\tau$. We now use a bootstrap argument to make the constants independent of time. For any $L\in \mathbb{N}$, we can write from \eqref{eqn:5} that
 \begin{align}\label{eqn:6}
 	\sup_{\tau\in\mathbb{N}}\left\lVert \sum_{i=1}^{I}u_i\right\rVert_{p, (\tau, \tau+L)\times\Omega}\le C(L).
 \end{align}
  Using \eqref{eqn:6} and \eqref{eqn:gc}, we get the estimate
  \begin{align}
  	\|u_i\|_{\frac{p}{\lambda}, (\tau, \tau+L)\times\Omega}\le C(L, \lambda)\|u_i\|_{p, (\tau, \tau+L)\times\Omega}\le C(L)
  \end{align}
  and
  \begin{align}
  	\|f_i(u)\|_{\frac{p}{\lambda}, (\tau, \tau+L)\times\Omega}\le C(L)\|u\|^\lambda_{\frac{p}{\lambda}, (\tau, \tau+L)\times\Omega}\le C(L).
  \end{align}
 Then application of the Lemma \ref{lemma:1} to the equation \eqref{eqn:uz} yields
  \begin{align}\label{eqn:uz1}
  	\|u_i\zeta_\tau\|^{(2,1)}_{\frac{p}{\lambda}, (\tau, \tau+L)\times\Omega}\le C(L)\|\zeta_\tau f_i(u)+\zeta_\tau'u_i\|_{\frac{p}{\lambda}, (\tau, \tau+L)\times\Omega}\le C(L).
  \end{align}
  Using embedding inequality of Lemma \ref{lemma:2}, we have
  \begin{align*}
  	\|u_i\zeta_\tau\|_{q, (\tau, \tau+L)\times\Omega}\le C(L)\|u_i\zeta_\tau\|^{(2,1)}_{\frac{p}{\lambda}, (\tau, \tau+L)\times\Omega}\le C(L), \text{ for all } q<p_1=\frac{(n+2)\frac{p}{\lambda}}{n+2-2\frac{p}{\lambda}}.
  \end{align*}
  That implies 
  \begin{align*}
  	\|u_i\|_{q, (\tau+1, \tau+L)\times\Omega}\le C(L)\text{ for all } q<p_1.
  \end{align*}
  Therefore, as long as $\frac{p_k}{\lambda}<\frac{n+2}{2}$ we get
  \begin{align*}
  		\|u_i\|_{q, (\tau+k, \tau+L)\times\Omega}\le C(L)\text{ for all } q<p_k.
  \end{align*}
 In this way, we have a recursive sequence $\{p_k\}$ where $p_{k+1}=\frac{(n+2)\frac{p_k}{\lambda}}{n+2-2\frac{p_k}{\lambda}}.$ Now since $p>\frac{(\lambda-1)(n+2)}{2}$, we can proceed similar to the proof of the existence of the classical solution and derive that $p_k\rightarrow \infty$ as $k\rightarrow\infty$. Hence there exists $k_0\in\mathbb{N}$ such that $\frac{p_{k_0}}{\lambda}>\frac{n+2}{2}$. Application of the maximal regularity result Lemma \ref{lemma:1} to
 \begin{align}
 	\frac{\partial (\zeta_{\tau+k_0}u_i)}{\partial t}-D_i\Delta (\zeta_{\tau+k_0}u_i)=\zeta^{'}_{\tau+k_0}u_i+\zeta_{\tau+k_0} f_i(u)
 \end{align} 
 yields
 \begin{align*}
 	\|\zeta_{\tau+k_0} u_i\|_{\infty, (\tau+k_0, \tau+L)}&\le C(L)	\|\zeta_{\tau+k_0} u_i\|^{(2, 1)}_{\frac{p_{k_0}}{\lambda}, (\tau+k_0, \tau+L)\times\Omega}\\
   &\le C(L)\|\zeta^{'}_{\tau+k_0}u_i+\zeta_{\tau+k_0} f_i(u)\|_{\frac{p_{k_0}}{\lambda}, (\tau+k_0, \tau+L)\times\Omega}
   \le C(L), \text{ by }\eqref{eqn:uz1}.
 \end{align*}
 This gives
 \begin{align*}
 	\|u_i\|_{\infty, (\tau+k_0+1, \tau+L)\times\Omega}\le C(L) \text{ for all }\tau\ge 0.
 \end{align*}
 Here the constant $C(L)$ is independent of $\tau$. Choosing $L=k_0+2$, we get
 \begin{align*}
 	\|u_i\|_{\infty, (j, j+1)\times\Omega}\le C(L).
 \end{align*}
 Consequently,
 \begin{align*}
 	\sup_{t\ge k_0+1}\|u_i(t)\|_{\infty, \Omega}\le C(L), \text{ for all } i=1, 2, . . .,I.
 \end{align*}
 Thus the unique global classical solution of the system $\eqref{eqn:M11}-\eqref{eqn:M13}$ is bounded uniformly in time.
 \section{Conclusion}
  We investigated a diffusion-reaction system and established the existence of global in time unique positive classical solution for the case of different diffusion coefficients. We also proved that the solution is bounded uniformly in time. This will help us to study the large-time behavior of the system. We didn't make any assumption about the conservation of the number of atoms or that the diffusion coefficients are close to each other therefore the result is applicable to a large class of reaction-diffusion systems. In our future works, we wish to study the case where the diffusion coefficients are required to be measurable and uniformly bounded and other possibilities of the diffusion coefficients along with nonlinear reaction rate terms.
  \section*{Data availability}
  Data sharing not applicable to this article as no datasets were generated or analyzed during the current study.
  \section*{Conflicts of interest} The authors declare no conflict of interest.
	\bibliographystyle{acm}
	\bibliography{gdre}
\end{document}